\documentclass{amsart}

\usepackage{amsmath, amssymb, amsthm, mathrsfs, graphicx}
\usepackage{hyperref, enumitem, xspace, ifthen,comment}
\usepackage[all]{xy}
\usepackage{tikz}
\usetikzlibrary{arrows,shapes,trees}

\newcommand{\mypagesize}{
\textwidth= 6.25in
\textheight=8.75in
\voffset-.5in
\hoffset-.75in
\marginparwidth=56pt
}
\mypagesize

\newtheorem*{thm-plain}{Theorem}
\newtheorem{thm}{Theorem}[section]
\newtheorem{lem}[thm]{Lemma}
\newtheorem{prp}[thm]{Proposition}
\newtheorem{cor}[thm]{Corollary}

\newtheorem{ques}[thm]{Question}
\numberwithin{equation}{thm}

\theoremstyle{definition}
\newtheorem{dfn}[thm]{Definition}

\newtheorem*{dfn-plain}{Definition}

\theoremstyle{remark}

\newtheorem{ntn}[thm]{Notation}
\newtheorem{rem}[thm]{Remark}

\newtheorem*{rem-plain}{Remark}

\DeclareMathOperator{\sing}{sing}

\DeclareMathOperator{\reg}{reg}

\DeclareMathOperator{\Spec}{Spec}

\DeclareMathOperator{\Pic}{Pic}
\DeclareMathOperator{\codim}{codim}

\DeclareMathOperator{\Exc}{Exc}

\DeclareMathOperator{\Bs}{Bs}

\def\rd#1.{\lfloor{#1}\rfloor}
\def\rp#1.{\lceil{#1}\rceil}

\newcommand{\N}{\mathbb N}

\newcommand{\Q}{\ensuremath{\mathbb Q}}

\newcommand{\C}{\mathbb C}

\renewcommand{\P}{\mathbb P}
\renewcommand{\O}{\mathscr O}

\newcommand{\x}{\times}

\renewcommand{\phi}{\varphi}
\renewcommand{\theta}{\vartheta}
\newcommand{\id}{\mathrm{id}}

\newcommand{\minus}{\setminus}
\newcommand{\inj}{\hookrightarrow}
\newcommand{\surj}{\twoheadrightarrow}

\newcommand{\isom}{\simeq}

\newcommand{\tensor}{\otimes}

\DeclareMathOperator{\sHom}{\mathscr H\!om}

\newcommand{\red}{\mathrm{red}}

\newcommand{\Sym}{\mathrm{Sym}}

\newcommand{\an}{\mathrm{an}}

\DeclareMathOperator{\pr}{pr}
\newcommand{\db}{Du~Bois\xspace}
\newcommand{\pdb}{potentially Du~Bois\xspace}
\newcommand{\mmp}{minimal model program\xspace}

\newcommand{\sI}{\mathscr{I}}

\newcommand{\sL}{\mathscr{L}}

\newcommand{\sO}{\mathscr{O}}

\newcommand{\sT}{\mathscr{T}}

\newdir{ ir}{{}*!/-5pt/@^{(}} 
\newdir{ il}{{}*!/-5pt/@_{(}} 

\DeclareMathOperator{\supp}{supp}

\newcommand{\eps}{\varepsilon}

\newcommand{\wt}{\widetilde}
\newcommand{\wh}{\widehat}
\newcommand{\kdot}{{{\,\begin{picture}(1,1)(-1,-2)\circle*{2}\end{picture}\ }}}

\newenvironment{sequation}{%
\setcounter{equation}{\value{thm}}%
\numberwithin{equation}{section}%
\begin{equation}%
}{%
\end{equation}%
\numberwithin{equation}{thm}%
\addtocounter{thm}{1}%
}

\numberwithin{equation}{thm}

\DeclareRobustCommand{\SkipTocEntry}[5]{}

\setitemize[1]{leftmargin=*,parsep=0em,itemsep=0.125em,topsep=0.125em}
\setenumerate[1]{leftmargin=*,parsep=0em,itemsep=0.125em,topsep=0.125em}

\newcommand{\iref}[3]{\the\value{#1}.\the\value{#2}(\the\value{#3})}

\newcommand\factor[2]{\left. \raise 2pt\hbox{$#1$} \right/\hskip -2pt \raise -2pt\hbox{$#2$}}

\definecolor{forrest}{RGB}{81,133,49}
\definecolor{mydarkblue}{RGB}{10,92,153}

\begin{document}

\title{Potentially Du~Bois spaces}
\author{Patrick Graf} %
\author{S\'andor J Kov\'acs} %
\address{PG: Lehrstuhl f\"ur Mathematik I, Universit\"at Bay\-reuth,
  95440 Bayreuth, Germany} %
\email{\href{mailto:patrick.graf@uni-bayreuth.de}{patrick.graf@uni-bayreuth.de}}
\address{SJK: University of Washington, Department of Mathematics, Box 354350,
  Seattle, WA 98195-4350, USA} %
\email{\href{mailto:skovacs@uw.edu}{skovacs@uw.edu}}
\date{\today} %
\thanks{The first named author was partially supported by the DFG-Forschergruppe 790
  ``Classification of Algebraic Surfaces and Compact Complex Manifolds''.}  %
\thanks{The second named author was supported in part by NSF Grant DMS-1301888, a
  Simons Fellowship (\#304043), and the Craig McKibben and Sarah Merner Endowed
  Professorship in Mathematics at the University of Washington. This work was
  partially completed while he enjoyed the hospitality of the Institute for Advanced
  Study (Princeton) supported by The Wolfensohn Fund.} %

\begin{abstract}
  We investigate properties of \emph{\pdb} singularities, that is, those that occur
  on the underlying space of a \db pair.
  We show that a normal variety $X$ with \pdb singularities and Cartier canonical
  divisor $K_X$ is necessarily log canonical, and hence \db. As an immediate
  corollary, we obtain the Lipman-Zariski conjecture for varieties with \pdb
  singularities.

  We also show that for a normal surface singularity, the notions of \db and \pdb
  singularities coincide.  In contrast, we give an example showing that in dimension
  at least three, a normal \pdb singularity $x \in X$ need not be \db even if one
  assumes the canonical divisor $K_X$ to be $\Q$-Cartier.
\end{abstract}

\maketitle


\section{Introduction}

Hodge theory of complex projective manifolds has proven to be an extremely
useful tool in many different and perhaps a priori unexpected situations. For
instance the simple consequence of the degeneration of the Hodge-to-de Rham spectral
sequence that the natural map
\begin{sequation}
  \label{eq:9}
  H^i(X^{\an}, \C)\to H^i(X^{\an}, \sO_{X^{\an}})
\end{sequation}%
is surjective has many applications. It was discovered early on that this
surjectivity continues to hold for normal crossing singularities and even some more
complicated singularities. Steenbrink identified the class of singularities that has
this property naturally and named them \emph{\db singularities} \cite{Steenbrink83}.
It turns out that (\ref{eq:9}) along with the requirement that general hyperplane
sections of \db singularities should also be \db essentially characterize \db
singularities \cite{MR2987664}.

Unfortunately, the rigorous definition of \db singularities is complicated. It relies
on a generalization of the de~Rham complex, the Deligne-\db complex (see
\cite[6.4]{SingBook}), an object in the derived category of coherent sheaves on $X$. However,
once the technical difficulties are settled the theory is very powerful.

One possible way to tame \db singularities is to consider this notion a weakening of
the notion of rational singularities. In fact, Steenbrink conjectured immediately
after introducing the notion that rational singularities are \db and this was
confirmed in \cite{Kov99}.

Originally, \db singularities were introduced to study degenerations of variations of
Hodge structures, but Koll\'ar noticed that there is a strong connection between them
and the singularities of the \mmp. In particular, he conjectured that log canonical
singularities are \db. This was recently confirmed in \cite{KK10}.
For the definition of the singularities of the \mmp, such as log canonical and klt,
please see~\cite[2.34]{KM98} or~\cite[2.8]{SingBook}.

The evolution of the \mmp taught us that singularities should be studied in pairs,
that is, instead of considering a single space $X$ one should consider a pair
consisting of a variety and a subvariety. This has also proved to be a powerful
generalization.

The notion of \db singularities was recently generalized for pairs $(X, \Sigma)$
consisting of a complex variety $X$ and a closed subscheme $\Sigma \subset X$ \cite{Kov11}. For a relatively detailed treatment the reader should peruse
\cite[Chapter 6]{SingBook}.

Recently \db singularities and \db pairs have provided useful tools in many
situations. For instance, they form a natural class of singularities where Kodaira
type vanishing theorems hold \cite{Ste85,KSS10,Kov11,Kov13,Pat13}. Other recent
applications include extension theorems~\cite{GKKP11}, positivity
theorems~\cite{Sch12}, categorical resolutions \cite{MR2981713}, log canonical
compactifications~\cite{HX13}, semi-positivity \cite{FSS13}, and injectivity theorems
\cite{Fujino_IT}. Besides applications in the \mmp, \db singularities play an
important role in moduli theory as well \cite{KK10, Kov13}.  

The introduction of \db pairs was motivated by the success of studying singularities
of pairs in the \mmp. If a pair $(X,\Delta)$ is log canonical (resp.\ klt) and
$\Delta$ is $\mathbb Q$-Cartier, then $X$ is also log canonical (resp.\ klt). An
analogous statement is not straightforward for \db pairs.  The fact that $(X,\Sigma)$
is a \db pair does not clearly imply that then $X$ itself is \db. In fact, one of the
advantages of \db pairs is that they provide a possibility to extend the power of
Hodge theoretic techniques to a larger class of varieties. In other words, it is
natural to ask and potentially helpful to know the answer to the following question.

\begin{ques}\label{ques:pdb} 
  Given a complex variety $X$, when does there exist a subvariety $\Sigma \subsetneq
  X$ such that $(X, \Sigma)$ is a \db pair?
\end{ques}

\noindent
To make it easier to discuss these singularities, we introduce the following
definition:

\begin{dfn}\label{dfn:pdb}
  Let $X$ be a complex variety and $x\in X$ a closed point. We say that $X$ is
  \emph{\pdb at $x$} if there exists a Zariski-open set $U\subseteq X$ containing $x$
  and a subvariety $\Sigma_U\subseteq U$ not containing any irreducible components of
  $U$ such that $(U,\Sigma_U)$ is a \db pair (see Section~\ref{sec:pair of DB} for
  the definition of \db pairs). $X$ is called \emph{\pdb} if it is \pdb at $x$ for
  every closed point $x\in X$.
\end{dfn}

\noindent
The main result of this paper gives some answers to Question~\ref{ques:pdb}. 

\begin{thm} \label{thm:pot DB} Let $X$ be a normal complex variety. Then:
\begin{enumerate}
\item\label{itm:pot DB.1} If $X$ has \pdb singularities, then it is \db in
  codimension two, i.e.~the non-Du Bois locus of $X$ has codimension at least three.
\item\label{itm:pot DB.2} Let $(X,\Sigma)$ be a \db pair and $\Delta\subset X$ a
  reduced effective divisor such that $\supp \Delta\subseteq \Sigma$ and $K_X+\Delta$
  is Cartier.  Then $(X,\Delta)$ is log canonical and hence $X$ is \db. In
  particular, if $X$ has \pdb singularities and $K_X$ is Cartier, then it is log
  canonical and \db.
\item\label{itm:pot DB.3}%
  There exists a three-dimensional normal variety $X$ with isolated singularities and
  \Q-Cartier canonical divisor such that $X$ has \pdb singularities, but it is not
  \db. In particular, the bound on the codimension in (\ref{thm:pot DB}.\ref{itm:pot
    DB.1}) is sharp.
\end{enumerate}
\end{thm}

\noindent
We have the following immediate corollary of Theorem~\ref{thm:pot DB}.\ref{itm:pot
  DB.2}.

\begin{cor}[Lipman-Zariski conjecture for \pdb singularities] \label{cor:LZ pdb}
Let $X$ be a complex variety with \pdb singularities. If the tangent sheaf
$\sT_X := \sHom_{\O_X}(\Omega_X^1, \O_X)$ is locally free, then $X$ is smooth.
\end{cor}

For normal singularities, Theorem~\ref{thm:pot DB}.\ref{itm:pot DB.1} and
\ref{thm:pot DB}.\ref{itm:pot DB.2} give a complete answer to Question~\ref{ques:pdb}
if either $K_X$ is Cartier or the dimension of $X$ is two.  Theorem~\ref{thm:pot
  DB}.\ref{itm:pot DB.3} can also be interpreted as a positive result: it says that
in higher dimensions, we stand a chance of finding a suitable $\Sigma \subset X$
making $(X, \Sigma)$ into a \db pair, even if $K_X$ is \Q-Cartier.

Let us now put Theorem~\ref{thm:pot DB} in a broader philosophical perspective.  It
is well known that for singularities in the \mmp, decreasing the boundary divisor by
a $\Q$-Cartier divisor improves the singularities of the pair. In particular, as
mentioned above, if $K_X$ is \Q-Cartier, then $(X, \Delta)$ being klt, dlt, lc,
etc.~implies that $X$ satisfies the same condition.
Theorem~\ref{thm:pot DB}.\ref{itm:pot DB.3} tells us that this principle does not
hold for \db singularities. It is relatively easy to give an example of a non-normal
singularity which is \pdb, but not \db \cite[Ex.~2.10]{KS13}. It is a little more
complicated to give a normal example of the same behavior \cite[Ex.~2.14]{KS13}. Then
it is much harder to give an example with a $\Q$-Cartier canonical divisor. Our
example in \ref{thm:pot DB}.\ref{itm:pot DB.3} does exactly that.

Another distinction between singularities of the \mmp and \db singularities is that
the former depend on the behavior of $mK_X$, a multiple of the canonical divisor for
a sufficiently divisible $m\in \N$, while rational and \db singularities depend on
the behavior of the canonical divisor itself. This distinction manifests itself in
the fact that rational singularities whose canonical divisor is Cartier are
canonical, but a rational singularity with only a $\Q$-Cartier canonical divisor does
not need to be even log canonical.  Theorem~\ref{thm:pot DB}.\ref{itm:pot DB.2} and
\ref{thm:pot DB}.\ref{itm:pot DB.3} demonstrate that this phenomenon also happens for
\db singularities.

\begin{rem}
  Our example in (\ref{thm:pot DB}.\ref{itm:pot DB.3}) has a \Q-Cartier canonical
  divisor of index $4$.  One might ask whether there also exist examples of index $2$
  or $3$. We conjecture that this is the case, however it is not clear whether our
  construction can be adapted to yield such examples.
\end{rem}

\subsection{Outline of proofs}

For Theorem~\ref{thm:pot DB}.\ref{itm:pot DB.1}, in the surface case we prove the
statement directly and then we conclude the general case using the deformation
invariance of \db singularities \cite{KS11}.  For Theorem~\ref{thm:pot
  DB}.\ref{itm:pot DB.2}, we apply the vanishing theorem of \cite[5.3]{Kov13}
together with the techniques from \cite[Section 17]{GKKP11} in order to obtain an
extension theorem for reflexive differentials, Theorem~\ref{thm:Ext of p-forms on DB
  pairs}, from which the claim immediately follows.

Theorem~\ref{thm:pot DB}.\ref{itm:pot DB.3} has the most involved proof and actually
the bulk of this paper is devoted to the construction of the example whose existence
is claimed there.  The basic idea is to take a cone over a smooth projective surface
$T$. So we first need a criterion for a cone (or a pair of cones) to be \db. It turns
out that such a criterion can be phrased in terms of the cohomology of certain line
bundles on $T$.  It is well-known that a cone over $T$ can have a \Q-Cartier
canonical divisor only if $K_T$ is either anti-ample, torsion, or ample. Furthermore,
in the first two cases, the cone is automatically \db, even log canonical. Hence our
example $T$ must necessarily be canonically polarized.

In order to construct $T$, first we find a ruled surface $S$ having the required
cohomological properties. Then $T$ is defined to be the general member of a suitable
linear system in the product $S \x B$, where $B$ is a curve of genus $2$. The tricky
part is to show that $T$ is smooth although it lives in a linear system with a
non-empty base locus.

\subsection{Acknowledgements}

We would like to thank Clemens J\"order, Stefan Kebekus, and Karl Schwede for
interesting discussions on the subject of this paper.
Furthermore we would like to thank the anonymous referee for providing suggestions
that made the paper more accessible.

\subsection{Notation, definitions, and conventions}

Throughout this paper, we work over the field of complex numbers $\C$.

A \emph{reduced pair} $(X, \Sigma)$ consists of $X$ a reduced scheme of finite type
over $\C$ and a reduced closed subscheme $\Sigma \subset X$.  
Neither $X$ nor $\Sigma$ is assumed to be pure dimensional or normal.
  
Let $(X,\Sigma)$ be a reduced pair and $x\in X$ a point. We say that $(X,\Sigma)$ is
\emph{snc at $x$} if there exists a Zariski-open neighbourhood $U\subseteq X$ of $x$
such that $U$ is smooth and $\Sigma \cap U$ is either empty, or a divisor with simple
normal crossings.  The pair $(X, \Sigma)$ is called an \emph{snc pair} or simply
\emph{snc} if it is snc at every point of $X$.

Given a reduced pair $(X,\Sigma)$, $(X,\Sigma)_{\reg}$ will denote the maximal open
subset of $X$ where $(X,\Sigma)$ is snc, and $(X,\Sigma)_{\sing}$ its complement,
with the induced reduced subscheme structure.

A reduced pair $(X,\Sigma)$ is called \emph{normal} if $X$ is normal and
\emph{projective} if $X$ is projective.

If $(X,\Sigma)$ is a normal reduced pair, then by definition $X$ is smooth in
codimension $1$.  Furthermore, since $\Sigma$ is reduced, near a general point of
$\Sigma$, both $X$ and $D$ are smooth. In particular in this case,
$\codim_X(X,\Sigma)_{\sing}\geq 2$, or in other words $(X,\Sigma)$ is snc in
codimension $1$.

Let $(X,\Sigma)$ be a normal reduced pair. A \emph{log resolution} of $(X,\Sigma)$ is
a proper birational morphism $\pi\!:\wt X\to X$ such that $\wt X$ is smooth,
both the pre-image $\pi^{-1}(\Sigma)$ of $\Sigma$ and the exceptional set
$E=\Exc(\pi)$ are of pure codimension $1$ in $\wt X$, and $(\wt X,\wt D+E)$ is an snc
pair where $\wt D= \pi^{-1}(\Sigma)_{\red}$ is the reduced divisor supported on
$\pi^{-1}(\Sigma)$.

Let $D,D_1,D_2$ be divisors on a normal variety. Then $D_1\vee D_2$ denotes the
smallest divisor that contains both $D_1$ and $D_2$ and $D_1\wedge D_2$ denotes the
largest divisor that is contained in both $D_1$ and $D_2$. Finally, we will use the shorthand $h^i(X, D)$ to denote $\dim_\C H^i(X,\sO_X(D))$.

\subsection{Du~Bois singularities and Du~Bois pairs}\label{sec:pair of DB}

The Deligne-Du~Bois complex \cite{DB81} associated to a complex variety $X$ is a
filtered complex $\underline\Omega_X^\kdot$, unique up to quasi-isomorphism, which
for a smooth $X$ is isomorphic to the de Rham complex considered with the stupid
filtration. Many of the usual cohomological properties of the de Rham complex that
hold for smooth varieties remain true for arbitrary varieties if one replaces the de
Rham complex with the Deligne-Du~Bois complex. In particular, if $X$ is proper, then
there is a Fr\"olicher-type spectral sequence converging to singular cohomology and
degenerating at $E_1$. We say that $X$ has Du~Bois singularities if the zeroth graded
piece of $\underline\Omega_X^\kdot$ is quasi-isomorphic to $\O_X$.  It follows that
if $X$ is proper and has Du~Bois singularities, then the natural map $H^i(X^\an, \C)
\to H^i(X^\an, \O_{X^\an})$ is surjective. In fact, this property is close to
characterizing Du~Bois singularities, cf.~\cite{MR2987664}. Note however that Du~Bois
varieties in general are not necessarily normal and hence may have singularities in
codimension 1.  For example normal crossing singularities are Du~Bois.

We briefly explain the construction of the (zeroth graded piece of the)
Deligne-Du~Bois complex. Given a singular variety $X$, first we resolve its
singularities by a log resolution $\pi$. However this is not enough, we also need to
resolve the singularities of the singular locus of $X$ and those of the exceptional
set of $\pi$, and then we need to resolve the exceptional sets of these resolutions,
and so on. To do this properly one ends up with a diagram of morphisms in the shape
of a $(\dim X + 1)$-dimensional hypercube, or more precisely a ``cubical
hyperresolution'' of $X$ cf.\ \cite{Carlson85,GNPP88,PetersSteenbrinkBook}.
Similarly to the way we associate a simple complex to a double complex, taking
disjoint unions of certain objects in this diagram yields a ``semi-simplicial variety
$X_\kdot$ with a morphism $\eps_\kdot\!: X_\kdot \to X$''. On every component of
$X_\kdot$, we may consider an injective resolution of the structure sheaf, and we may
put all these resolutions together into a double complex using the pull-back and
push-forward maps between the components of $X_\kdot$.  Now by applying
${\eps_\kdot}_*$ to this double complex and forming the associated simple complex, we
obtain the derived push-forward $R{\eps_\kdot}_* \sO_{X_\kdot}$.  This is exactly the
zeroth graded piece of the Deligne-Du~Bois complex. For a more detailed, yet still
down-to-earth introduction see \cite[\S2]{Ste85}, for a rigorous treatment see
\cite{DB81,GNPP88} or the more recent \cite[\S7.3]{PetersSteenbrinkBook}.

Du~Bois pairs were introduced in \cite{Kov11}. Their definition is as involved as the
definition of Du~Bois singularities, so we will not repeat it here. The essential
ingredient is the following: For any reduced pair $(X,\Sigma)$ there exists an object
in the filtered derived category of $X$, called the Deligne-Du~Bois complex of
$(X,\Sigma)$ that relates the Deligne-Du~Bois complex of $X$ and that of $\Sigma$.
If $(X,\Sigma)$ is an snc pair, then the Deligne-Du~Bois complex of $(X,\Sigma)$ is
quasi-isomorphic to $\Omega_X^\kdot(\log \Sigma)(-\Sigma)$.  A reduced pair
$(X,\Sigma)$ is a \emph{Du~Bois pair} if the associated zeroth graded quotient of the
Deligne-Du~Bois complex of $(X,\Sigma)$ is quasi-isomorphic to the ideal sheaf of
$\Sigma$ in $\O_X$. 
For more on Du~Bois pairs see \cite{Kov11} and \cite[\S6]{SingBook}.

As we already defined in Definition~\ref{dfn:pdb}, a variety $X$ is said to have
\emph{\pdb} singularities if there exists a cover of $X$ by Zariski-open subsets $U_i
\subset X$ such that for any $i$, there exists a subvariety $\Sigma_i \subset U_i$
not containing any irreducible component of $U_i$ and making $(U_i, \Sigma_i)$ into a
\db pair.

\section{Two Du~Bois criteria}\label{sec:cone DB pair}

In this section, we give two necessary and sufficient criteria for pairs of a certain
kind to be Du~Bois.  The first criterion is concerned with varieties with only
isolated non-Du~Bois points, while the second one, a corollary of the first, deals
with cones over Du~Bois pairs $(X, \Sigma)$.  The latter criterion is likely known to
experts in some form at least in the case $\Sigma = \emptyset$. In that case a
similar statement was proved in \cite[Thm.~4.4]{Ma13}.

\begin{thm}[A Du~Bois criterion for isolated non-\db locus]\label{thm:isol DB
    criterion} Let $(X, \Sigma)$ be a normal reduced pair with a closed point $x \in
  X$ such that $(X\setminus \{x\}, \Sigma\setminus\{x\})$ is a \db pair.  Let $f\!: Y
  \to X$ be a proper birational morphism from a normal variety $Y$ that is an
  isomorphism over $X\setminus\{x\}$.  Let $E$ denote the (not necessarily divisorial)
  exceptional locus of $f$, and set $\Gamma = E \cup f^{-1}(\Sigma)$. Assume that
  $(Y, \Gamma)$ is a Du~Bois pair.  Then $(X, \Sigma)$ is a Du~Bois pair if and only
  if
  \begin{equation}\label{eqn:isol DB criterion}
    R^i f_* \sO_Y(-\Gamma) = 0
  \end{equation}%
  for all $i \ge 1$.
\end{thm}

\subsection{Cones over pairs}\label{subsec:cones}

First we recall some basic facts about cones, and we fix the notation used in
Theorem~\ref{thm:cone DB criterion}. We will follow the conventions and notation 
of \cite[\S3.1]{SingBook}.

\begin{ntn}[Affine cones]\label{ntn:cones}
  Let $X$ be a projective scheme and $\sL \in \Pic X$ an ample line bundle on $X$.
  The \emph{affine cone over $X$ with respect to $\sL$} is
  \[ 
  C_a(X, \sL) = \Spec R(X, \sL), 
  \]
  where
  \[ 
  R(X, \sL) = \bigoplus_{n \ge 0} H^0(X, \sL^n) 
  \] 
  is the section ring of $\sL$.  If there is no ambiguity about the choice of $\sL$,
  we will write $CX$ for $C_a(X, \sL)$.  If $X$ is connected, the \emph{vertex} $P
  \in CX$ is defined to be the closed point corresponding to the maximal ideal %
  \[ 
  \bigoplus_{n \ge 1} H^0(X, \sL^n) \subset R(X, \sL). 
  \]
\end{ntn}

\begin{rem}
  See \cite[\S3.1]{SingBook} for generalities on cones. Note that this construction
  works even if $\sL$ is not very ample.  If $X$ is normal and $\sL$ is very ample,
  then $C_a(X, \sL)$ is the normalization of the classical affine cone over the
  embedding of $X$ via $\sL$. In particular, $C_a(X, \sL)$ is isomorphic to the
  classical affine cone if and only if the embedding given by $\sL$ is projectively
  normal.  Notice further that $C_a(X, \sL)$ is normal whenever $X$ is normal, even
  if $\sL$ is not very ample.
\end{rem}

\begin{ntn}[Pairs of cones]
  Let $X$ be a normal projective variety, $\Sigma \subset X$ a reduced subscheme, and
  $\sL \in \Pic X$ an ample line bundle on $X$. There is a natural map $\iota\!:
  C_a(\Sigma, \sL|_\Sigma) \to C_a(X, \sL)$, which is a closed embedding away from
  the vertex $P \in CX$.  We will also denote $C_a(X, \sL)$ by $CX$, the image of
  $\iota$ by $C\Sigma$, and the \emph{pair of cones} consisting of $CX$
  and $C\Sigma$ by $C_a(X, \Sigma, \sL)$, or for simplicity by $(CX, C\Sigma)$.
\end{ntn}

\begin{thm}[A Du~Bois criterion for cones]\label{thm:cone DB criterion} 
  Let $(X,\Sigma)$ be a normal projective \db pair and $\sL$ an ample line bundle on
  $X$.  Then the pair of cones $C_a(X, \Sigma, \sL)$ is a Du~Bois pair if and only if
  \begin{sequation}\label{eqn:cone DB criterion}
    H^i(X, \sL^n(-\Sigma)) = 0 
  \end{sequation}%
  for all $i, n \ge 1$.
\end{thm}

\subsection{Proof of Theorem~\ref{thm:isol DB criterion}}

First assume that $\Sigma = \emptyset$. In this case, $\Gamma = E$.  Consider the
following commutative diagram of distinguished triangles:
\[
\xymatrix{%
  \underline\Omega_{X, \{x\}}^0 \ar@{-->}[d]_\alpha\ar[r] &
  \underline\Omega_{X}^0 \ar[r]\ar[d] & \underline\Omega_{\{x\}}^0 \ar[d]\ar[r]^-{+1} & \\
  Rf_* \underline\Omega_{Y,E}^0 \ar[r] & Rf_* \underline\Omega_Y^0 \ar[r] & Rf_*
  \underline\Omega_E^0 \ar^-{+1}[r] & , }%
\]
where the solid arrows are the obvious natural maps and $\alpha$ is the induced map
that keeps the diagram commutative cf.~\cite[Prop.~3.11]{Kov11}.  Next consider the
distinguished triangle from \cite[Prop.~4.11]{DB81}:
\[
\xymatrix{%
  \underline\Omega_{X}^0 \ar[r] & \underline\Omega_{\{x\}}^0 \oplus Rf_*
  \underline\Omega_Y^0 \ar[r] & Rf_* \underline\Omega_E^0 \ar^-{+1}[r] & , }%
\]
and observe that combined with \cite[(2.1.4)]{KK10} this implies that $\alpha$ is an
isomorphism and hence there exists the following distinguished triangle:
\[
\xymatrix{%
  Rf_* \underline\Omega_{Y,E}^0 \ar[r] & \underline\Omega_{X}^0 \ar[r] &
  \underline\Omega_{\{x\}}^0 \ar^-{+1}[r] & .  }%
\]

Since $(Y,E)$ is a Du~Bois pair by assumption, $\underline\Omega_{Y,E}^0\simeq
\sO_Y(-E)$ and since $X$ is normal (for this seminormal would be enough), it follows
by \cite[Prop.~5.2]{MR1741272} (cf.~\cite[5.6]{MR2503989}, \cite[7.6]{KS09}) that
$h^0(\underline\Omega_{X}^0)\simeq \sO_{X}$ and hence we have the following long
exact sequence:
\[
0 \to f_*\sO_Y(-E) \to \sO_{X} \to \sO_{\{x\}} \to R^1f_*\sO_Y(-E) \to
h^1(\underline\Omega_{X}^0) \to \underbrace{ h^1(\underline\Omega_{\{x\}}^0) }_{=\,0}
\to \dots
\]
Since $\sO_{X}\to \sO_{\{x\}}$ is surjective, we conclude that $R^if_*\sO_Y(-E)\simeq
h^i(\underline\Omega_X^0)$ for $i\geq 1$ and hence in this case $X$ (or equivalently
$(X, \Sigma)$) is Du~Bois if and only if
\begin{sequation}
 \label{eq:1}
  R^if_*\sO_Y(-E)=0 \qquad \text{for $i \ge 1$}
\end{sequation}%

We will now turn to the case $\Sigma \neq \emptyset$.  In particular, since the
statement is local, we may assume that $x \in \Sigma$.  Again by
\cite[Prop.~4.11]{DB81}, we have a diagram of distinguished triangles:
\[
\xymatrix@R=1.5em{%
  \underline\Omega_{X,\Sigma}^0 \ar^-{\protect{\gamma}}[r] \ar[d] &
  Rf_* \underline\Omega_{Y,\Gamma}^0 \ar[d] \\
  \underline\Omega_{X}^0 \ar[r] \ar[d] & \underline\Omega_{\{x\}}^0 \oplus Rf_*
  \underline\Omega_Y^0 \ar[r] \ar[d] & Rf_* \underline\Omega_E^0 \ar^-{+1}[r]
  \ar@{=}[d] & \\
  \underline\Omega_{\Sigma}^0 \ar[r] \ar^-{+1}[d] & \underline\Omega_{\{x\}}^0 \oplus
  Rf_* \underline\Omega_\Gamma^0 \ar[r] \ar^-{+1}[d] & Rf_* \underline\Omega_E^0
  \ar^-{+1}[r] & . \\
  &&& \\
}%
\]
It follows from \cite[(B.1.1)]{Kovacs11b} that $\gamma$ is an isomorphism.  Since we
assumed $(Y, \Gamma)$ to be a Du~Bois pair, we obtain an isomorphism
\[
\underline\Omega_{X,\Sigma}^0 \isom Rf_* \O_Y(-\Gamma).
\] 
So $(X, \Sigma)$ is a Du~Bois pair if and only if
\[ 
f_* \O_Y(-\Gamma) \isom \sI_{\Sigma \subset X} \text{\,and\,} R^i f_*
\O_Y(-\Gamma) = 0 \text{ for $i \ge 1$}.
\]
Observe that there exists a short exact sequence:
\[
\xymatrix{%
  0 \ar[r] & \sI_{\Sigma \subset X} \ar[r] \ar[d] & \sO_{X} \ar[r] \ar[d]_ {\simeq}
  ^{ \underset{\text{\tiny is normal}} {\text{\tiny since $X$}}} & \sO_{\Sigma}
  \ar@<-1ex>@{^(->}[d] ^-{ \underset{\text{\tiny and $\Sigma=f(\Gamma)$}}
    {\underset{\text{\tiny is reduced}} {\text{\tiny since $\Sigma$}}}}
  \ar[r] & 0 \\
  0\ar[r] & f_*\sO_Y(-\Gamma) \ar[r] & f_*\sO_Y \ar[r] & f_*\sO_{\Gamma} \ar[r] &
  R^1f_*\sO_Y(-\Gamma) \ar[r] & \dots \\
}
\]
It follows that the image of $f_*\sO_Y$ in $f_*\sO_{\Gamma}$ is exactly
$\sO_{\Sigma}$ and hence $f_* \O_Y(-\Gamma) \simeq \sI_{\Sigma \subset X}$ always
holds.
So we obtain that $(X, \Sigma)$ is a Du~Bois pair if and only if
\begin{sequation}
  \label{eq:2}
  R^i f_* \sO_Y(-\Gamma) = 0 \text{ for $i \ge 1$}.
\end{sequation}%
  
Notice that if $\Sigma=\emptyset$, then $\Gamma=E$ so the conditions \eqref{eq:1} and
\eqref{eq:2} are actually the same whether $\Sigma=\emptyset$ or
$\Sigma\neq\emptyset$. This proves Theorem~\ref{thm:isol DB criterion}.

\subsection{Proof of Theorem~\ref{thm:cone DB criterion}}
Let $f\!: Y =\Spec_X \bigoplus _{m\geq 0} \sL^m \to CX$ be a weighted blowup of the
vertex $P \in CX$, with exceptional divisor $E \subset Y$ (cf.\
\cite[p.98]{SingBook}).  Then $Y$ is the total space of the dual bundle $\sL^{-1}$
and hence the natural map $\pi\!: Y \to X$ is a smooth affine morphism. Let $Z :=
\pi^{-1}(\Sigma)$ and $\Gamma := Z \cup E$.

\begin{lem} \label{lem:826}
  $(Y,\Gamma)$ is a \db pair.
\end{lem}

\begin{proof}
  Since $(X,\Sigma)$ is a \db pair, it follows from \cite[6.19]{SingBook} that
  $(Y,Z)$ is a \db pair and from \cite[6.17]{SingBook} that $(\Gamma,Z)$ is a \db
  pair. Hence, the second and third rows of the following diagram form distinguished
  triangles.  
  $$
  \xymatrix@R=1.5em@C=5em{%
    \sI_{\Gamma\subset Y} \ar[r]\ar[d] & \underline\Omega^0_{Y,\Gamma} \ar[d] & \\
    \sI_{Z\subset Y} \ar[r]\ar[d] & \underline\Omega^0_{Y} \ar[d] \ar[r] &
    \underline\Omega^0_Z \ar[r]^-{+1} \ar[d]^{\id}&  \\
    \sI_{Z\subset \Gamma} \ar[r] \ar[d]^-{+1} & \underline\Omega^0_{\Gamma}
    \ar[d]^-{+1} \ar[r] &
    \underline\Omega^0_Z \ar[r]^-{+1}&  \\
    && }
  $$
  The first column is an obvious short exact sequence of ideals and the second is a
  distinguished triangle by definition. Therefore, \cite[(B.1.1)]{Kovacs11b} implies
  that $\underline\Omega^0_{Y,\Gamma}\simeq \sI_{\Gamma\subset Y}$, so indeed
  $(Y,\Gamma)$ is \db.
\end{proof}

\begin{proof}[{Proof of Theorem~\ref{thm:cone DB criterion} continued}]
  By Lemma~\ref{lem:826} and Theorem~\ref{thm:isol DB criterion}, $(CX,C\Sigma)$ is a
  \db pair if and only if $R^i f_* \O_Y(-\Gamma)=0$ for all $i \ge 1$, so we need to
  prove that this vanishing is equivalent to~\eqref{eqn:cone DB criterion}.

First notice that $R^i f_* \O_Y(-\Gamma)$ is a skyscraper sheaf supported on $P \in
CX$, with stalk
\begin{sequation}
  \label{eq:7}
  H^i(Y, \O_Y(-\Gamma)) \isom H^i(X, \pi_*\O_Y(-\Gamma)),
\end{sequation}%
where the isomorphism follows because $\pi$ is an affine morphism.  In the remainder
of the proof we will demonstrate that
\begin{sequation}
  \label{eq:3}
  \pi_*\sO_Y(-\Gamma)\simeq \bigoplus_{n\geq 1} \sL^n(-\Sigma)
\end{sequation}%
which, combined with (\ref{eq:7}), implies the desired statement.

Recall that $\Gamma=Z\cup E$ and $\pi$ is affine and
consider the following diagram of short exact sequences:
\begin{sequation}
  \label{eq:5}
  \xymatrix@R=1.5em{%
    & 0 \ar[d]     & 0 \ar[d]     & 0 \ar[d] & \\
    0 \ar[r] & \pi_*\sO_Y(-\Gamma) \ar[r] \ar[d] & \pi_*\sO_Y(-Z) \ar[r] \ar[d] &
    \pi_*\sO_{E}(-Z|_E) \ar[r] \ar[d] & 0 \\
    0 \ar[r] & \pi_*\sO_Y(-E) \ar[r] \ar[d] & \pi_*\sO_Y \ar[r] \ar[d] &
    \pi_*\sO_{E} \ar[r] \ar[d] & 0 \\
    0 \ar[r] & \pi_*\sO_Z(-E|_Z) \ar[r] \ar[d] & \pi_*\sO_Z \ar[r] \ar[d] &
    \pi_*\sO_{Z\cap E} \ar[r] \ar[d] & 0 \\
    & 0 & 0 & 0 }
\end{sequation}%
By construction
\[ 
\pi_*\sO_Y \simeq \bigoplus_{n \ge 0} \sL^n, \quad\text{ and }\quad \
\pi_*\sO_Z \simeq \bigoplus_{n \ge 0} \sL^n|_\Sigma, 
\]
so 
\begin{sequation}
  \label{eq:4}
  \pi_*\sO_Y(-Z) \simeq \bigoplus_{n \ge 0} \sL^n(-\Sigma). 
\end{sequation}%
It is easy to see that $\pi$ induces isomorphisms $E\simeq X$ and $Z\cap E\simeq
\Sigma$, and hence
\[
\pi_*\sO_E\simeq \sO_X \qquad\text{and} \qquad \pi_*\sO_{Z\cap E}\simeq \sO_\Sigma,
\]
which implies that
\begin{sequation}
  \label{eq:6}
  \pi_*\sO_E(-Z|_E) \simeq \sO_X(-\Sigma)\simeq \sL^0(-\Sigma).
\end{sequation}%
Finally \eqref{eq:5}, \eqref{eq:4}, and \eqref{eq:6} together imply \eqref{eq:3} and
hence Theorem~\ref{thm:cone DB criterion} follows.
\end{proof}

\section{Proof of Theorem~(\ref{thm:pot DB}.\ref{itm:pot DB.1})}

\begin{prp}\label{prp:DB surface pair}
Let $(X, \Sigma)$ be a Du Bois pair, where $X$ is a normal surface. Then $X$ is Du Bois.
\end{prp}

\begin{proof}
  We may assume that $X$ has a single isolated singularity $x \in X$, that $\Sigma$
  is a divisor and that $x \in \Sigma$.  Let $\pi\!: Y \to X$ be a log resolution of
  $(X, \Sigma)$, with exceptional set $E \subset Y$, and let $\Gamma \subset Y$ be
  the preimage of $\Sigma$.  Notice that $\Gamma = E + T$, where $T := \pi^{-1}_*
  \Sigma$ is the strict transform of $\Sigma$.  

  By Theorem~\ref{thm:isol DB criterion} applied to $(X, \emptyset)$, we have that
  $X$ is Du Bois if and only if
  \begin{equation}\label{eq:surface DB} 
    R^1 \pi_* \sO_Y(-E) = 0,
  \end{equation}
  while $(X, \Sigma)$ is a Du Bois pair if and only if
  \begin{equation}\label{eq:surface pair DB} 
    R^1 \pi_* \sO_Y(-\Gamma) = 0.
  \end{equation}
  So it suffices to show that \eqref{eq:surface pair DB} implies \eqref{eq:surface
    DB}. To this end, consider the short exact sequence
  \[ 
  0 \to \sO_Y(-\Gamma) \to \sO_Y(-E) \to \sO_T(-E) \to 0 
  \] 
  and apply $\pi_*$. The associated long exact sequence gives the following:
  \[ 
  \dots\to R^1 \pi_* \sO_Y(-\Gamma) \to R^1 \pi_* \sO_Y(-E) \to R^1 \pi_* \sO_T(-E)
  \to \dots.
  \]
  However, $R^1 \pi_* \sO_T(-E) = 0$ because $\pi|_T$ is a finite morphism,
  and then \eqref{eq:surface pair DB} implies the desired statement.
\end{proof}

\begin{proof}[Proof of Theorem \protect{(\ref{thm:pot DB}.\ref{itm:pot DB.1})}]
  After shrinking $X$, we may assume that $X$ is affine and that there is
  a subvariety $\Sigma \subset X$ such that $(X, \Sigma)$ is a Du Bois pair.
  Let $H \subset X$ be the
  intersection of $n - 2$ general hyperplanes in $X$, where $n$ is the dimension of $X$. Then $(H, \Sigma|_H)$ is a Du
  Bois pair by repeated application of Lemma \ref{lem:Cutting down DB pairs}. By
  Proposition \ref{prp:DB surface pair}, it follows that $H$ is Du Bois.  A repeated
  application of \cite[Thm.~4.1]{KS11} now shows that $X$ is Du Bois near $H$. So the
  non-Du Bois locus of $X$ does not intersect a general complete intersection surface
  in $X$. It follows that the non-Du Bois locus has codimension at least three.
\end{proof}

\section{Proof of Theorem~(\ref{thm:pot DB}.\ref{itm:pot DB.2})}

Theorem~(\ref{thm:pot DB}.\ref{itm:pot DB.2}) is an immediate consequence of the
following theorem, which for log canonical pairs was proved in \cite[Theorem
16.1]{GKKP11}.

\begin{thm}[Extension theorem for $p$-forms on Du Bois pairs]\label{thm:Ext of
    p-forms on DB pairs} Let $(X, \Sigma)$ be a Du~Bois pair, where $X$ is normal. If
  $\pi\!: \wt X \to X$ is a log resolution of $(X, \Sigma)$ with exceptional divisor
  $E=\Exc(\pi)$, then the sheaves
  \[ 
  \pi_* \Omega_{\wt X}^p(\log \wh D), \quad 0 \le p \le n, 
  \] 
  are reflexive, where $\wh D=(\pi^{-1}(\Sigma)\vee E)_{\red}$, the reduced divisor
  with support $\pi^{-1}(\Sigma) \cup \supp E\subset \wt X$.
\end{thm}

\begin{proof}[{Proof of Theorem~(\ref{thm:pot DB}.\ref{itm:pot DB.2})}]
We will use the notation from Theorem~\ref{thm:Ext of p-forms on DB pairs} above.
Write $\Sigma=\Sigma_{\text{div}}\cup \Sigma_{\text{non-div}}$ as the union of closed
sets such that $\Sigma_{\text{div}}$ is a divisor and $\codim_X
\Sigma_{\text{non-div}}\geq 2$. Let $Z\:= \pi(E)\cup \Sigma_{\text{non-div}}$ and
$U\:=X\setminus Z$.  It follows that $\big(\pi_*\omega_{\wt X} (\wh D)\big)|_U \simeq
\big(\omega_X(\Sigma_{\text{div}})\big)|_U$, and applying Theorem~\ref{thm:Ext of
  p-forms on DB pairs} in the case $p = \dim X$ we obtain that $\pi_*\omega_{\wt X}
(\wh D)$ is reflexive.  Since $X$ is normal, $\codim_XZ\geq 2$, and hence
$\pi_*\omega_{\wt X} (\wh D) \simeq \omega_X(\Sigma_{\text{div}})$.

Now, by assumption $\supp \Delta\subset \Sigma$ and $K_X+\Delta$ is Cartier.
Therefore $\omega_X(\Delta)\subseteq \omega_X(\Sigma_{\text{div}})\simeq
\pi_*\omega_{\wt X} (\wh D)$. In other words there exists a non-zero morphism
$\omega_X(\Delta)\to \pi_*\omega_{\wt X} (\wh D)$ which is an isomorphism on a
non-empty open subset of $X$. By adjointness this implies the existence of a non-zero
morphism $\pi^*\omega_X(\Delta)\to \omega_{\wt X} (\wh D)$ which is an isomorphism on
a non-empty open subset of $\wt X$. Since $\omega_X(\Delta)$ is a line bundle, this
means that $\pi^*\omega_X(\Delta)\to \omega_{\wt X} (\wh D)$ is actually injective,
which means that all the discrepancies of the pair $(X,\Delta)$ are at least $-1$,
that is, $(X,\Delta)$ is log canonical. 
\end{proof}

\subsection{Steenbrink-type vanishing results}\label{subsec:Steenbrink}

Next we turn to the proof of Theorem~\ref{thm:Ext of p-forms on DB pairs}.  In fact,
the argument used to prove \cite[Thm.~16.1]{GKKP11} works in this case essentially
unchanged, provided the ingredients of that proof are adapted to the present
situation.

The proof of Theorem~\ref{thm:Ext of p-forms on DB pairs} relies on Steenbrink-type
vanishing results which, again, were already proved in \cite[Sec.~14]{GKKP11} for log
canonical pairs. Here we need the following more general statement:

\begin{thm}[{\cite[5.3]{Kov13}, cf.\ \cite[Theorem
    14.1]{GKKP11}}]\label{thm:Steenbrink I}
  Let $(X, \Sigma)$ be a Du~Bois pair and assume that $\dim X\geq 2$. Further let
  $\pi\!: \wt X \to X$ be a log resolution of $(X, \Sigma)$ with exceptional divisor
  $E=\Exc(\pi)$ and let $\wh D=(\pi^{-1}(\Sigma)\vee E)_{\red}$.  
  Then we have
  \[ 
  R^{n-1} \pi_* \big( \Omega_{\wt X}^p(\log \wh D) \tensor \O_{\wt X}(-\wh D) \big) =
  0
  \]
  for all $0 \le p \le n$. \qed
\end{thm}

\begin{cor}[cf.\ {\cite[Corollary 14.2]{GKKP11}}]\label{cor:Steenbrink II}
  Let $(X, \Sigma)$ be a Du~Bois pair and assume that $\dim X\geq 2$. Further let
  $\pi\!: \wt X \to X$ be a log resolution of $(X, \Sigma)$ with exceptional divisor
  $E=\Exc(\pi)$, $\wh D=(\pi^{-1}(\Sigma)\vee E)_{\red}$,
  and $x \in X$ a point with reduced fibre $F_x = \pi^{-1}(x)_\red$. Then
  \[ 
  H^1_{F_x} \big( \wt X, \Omega_{\wt X}^p(\log \wh D) \big) = 0 \quad \text{for all }
  0 \le p \le n. 
  \]
\end{cor}

\begin{proof}
  This follows from Theorem~\ref{thm:Steenbrink I} by applying duality for cohomology
  with support \cite[Theorem A.1]{GKK10}.
\end{proof}

\subsection{Proof of Theorem~\ref{thm:Ext of p-forms on DB pairs}}

We will need the following technical lemma.

\begin{lem}[Bertini theorem for Du~Bois pairs]\label{lem:Cutting down DB pairs}
  Let $(X, \Sigma)$ be a Du~Bois pair, $H \in |L|$ a general member of a
  basepoint-free linear system, and $\Sigma_H := \supp(\Sigma \cap H)$. Then $(H,
  \Sigma_H)$ is also a Du~Bois pair.
\end{lem}

\begin{proof}
  This is proved in \cite[3.18]{Kov11} (cf.\ \cite[6.5.6]{SingBook}).
\end{proof}

\begin{proof}[Proof of Theorem~\ref{thm:Ext of p-forms on DB pairs}]
  We mainly follow the proof given in \cite[Section 17]{GKKP11} with some
  adjustements. Here we explain the main ideas of that proof with the necessary
  changes. The reader is referred to \cite{GKKP11} for technical details. 

  The proof works by proving the extension statement one-by-one over the irreducible
  components $E_0$ of the exceptional locus $E$ of a log resolution $\pi\!: \wt X \to X$
  of $X$. The argument follows a double induction on the dimension of $X$ and the
  codimension of the image $\pi(E_0) \subset X$.
  There are two main techniques used in the proof. The first one is extending sections from an open
  set to an ambient set by using vanishing of the local cohomology group that connects
  the two. This is exactly what is provided by Theorem~\ref{thm:Steenbrink I} and
  Corollary~\ref{cor:Steenbrink II} which replace \cite[Corollary 14.2]{GKKP11} in
  the original proof. The other main tool is cutting by hyperplane sections and in
  order for that to be effective we need a Bertini type statement. This is provided
  by Lemma~\ref{lem:Cutting down DB pairs}.

  Following the proof in \cite[Section 17]{GKKP11} the first issue we need to deal
  with is that here $\Sigma$ is not assumed to be a divisor. This is however not a
  real problem. In the original proof $\Sigma$ only needs to be a divisor so the
  notion of being log canonical make sense. It is never used that $\Sigma$ is a
  divisor in any other way, so the arguments still make sense if one replaces the
  words ``log canonical'' with ``Du~Bois pair''.

  For the start of the induction, that is $\dim X=2$, we use
  Corollary~\ref{cor:Steenbrink II}. As mentioned above, the point of this step as
  well as the heart of the entire proof is that we are able to extend sections from
  an open set to an ambient set if we have a vanishing of the local cohomology group
  that connects the two. 

  Similarly, the setup and the simplifications performed in the inductive step work fine until we
  need to use the vanishing of the appropriate local cohomology groups. There we need
  to substitute Corollary~\ref{cor:Steenbrink II} for \cite[Corollary 14.2]{GKKP11}
  in the original argument.  The same needs to be done with all further occurrences
  of that corollary. 

  We also need to replace the Bertini type statement of \cite[Lemma 2.23]{GKKP11}
  with Lemma~\ref{lem:Cutting down DB pairs} throughout the proof. In particular, the
  use of \cite[Claim~17.15]{GKKP11} needs to be modified so as to read ``If $t \in T$
  is a general point, then $(X_t, \Sigma_t)$ is a Du~Bois pair''. 

  The rest of the proof goes through without any change.
\end{proof} 

\section{Preparation for the proof of Theorem~(\ref{thm:pot DB}.\ref{itm:pot DB.3})}

\subsection{Split ruled surfaces}

We recall some basic facts about surfaces that arise as the projectivization of a
split rank two vector bundle over a curve.  
Let us start by fixing notation.

\begin{ntn}\label{ntn:split ruled surface} 
  Let $C$ be a smooth projective curve, and let $A$ be a divisor on $C$. Form the
  ruled surface
  \[ 
  \pi\!: S = \P_C(\O_C \oplus \O_C(-A)) \to C,
  \] 
  and let $\sO_S(1)$ denote the associated relatively ample line bundle.  Let $E
  \subset S$ be the section of $\pi$ corresponding to the surjection $\O_C \oplus
  \O_C(-A) \surj \O_C(-A)$, and let $E_\infty \subset S$ be the section corresponding
  to the projection onto the first summand $\O_C \oplus \O_C(-A) \surj \O_C$.
\end{ntn}

As $\pi$ induces an isomorphism from $E$ and from $E_\infty$ onto $C$, we will
identify divisors on $E$, on $E_\infty$, and on $C$ with their images and pre-images
via $\pi$. Notice that $E$ is contained in the linear system $|\O_S(1)|$, and that
$\O_S(E)|_E \simeq \O_E(-A)$.

\begin{prp}[Divisors on a ruled surface]\label{prp:div ruled}
  Let $M$ be a divisor on $C$. For any $n \ge 1$, the projection formula yields an
  isomorphism
  \[ 
  \phi\!: \bigoplus_{k=0}^n H^0(C, \sO_C(M - kA)) \to H^0(S, \sO_S(\pi^* M + nE)).
  \] 
  For some integer $k$, $0 \le k \le n$, let $0 \ne s \in H^0(C, \sO_C(M - kA))$ be a
  nonzero section, with divisor $D$. Then the divisor of $\phi(s) \in H^0(S,
  \sO_S(\pi^* M + nE))$ is
  \[ 
  \pi^* D + (n - k) E + k E_\infty. 
  \]
\end{prp}

\begin{proof}
We prove this proposition in three steps.

Step 1: $M = 0$, $n = 1$.
In this case, the claim follows immediately from Lemma~\ref{lem:P(V)} below.

Step 2: $M = 0$, $n$ arbitrary.
Since we have
\[ 
\pi_* \O_S(n) = \Sym^n \big( \O_C \oplus \O_C(-A) \big) = \bigoplus_{k=0}^n
\O_C(-kA),
\] 
this case follows from Step 1.

Step 3: $M$ and $n$ arbitrary.
The isomorphism in the projection formula is given by
\begin{equation*}
  \xymatrix@R=.05ex{%
    \O_C(M) \tensor \pi_* \O_S(n) \ar[r] & \pi_* \big( \pi^*\! \O_C(M) \tensor
    \O_S(n)\big) \\
    s \tensor t\  \ar@{|->}[r] & \pi^*\! \ s \tensor t.\\ }
\end{equation*}
So this general case follows from Step 2.
\end{proof}

\begin{lem}[Hyperplanes in projective space]\label{lem:P(V)}
Let $V$ be a finite-dimensional complex vector space, and let $\P(V)$ be the space of one-dimensional quotients of $V$. We have a canonical isomorphism
\[ \phi\!: V \to H^0(\P(V), \O_{\P(V)}(1)). \]
For any $0 \ne v \in V$, the divisor of the section $\phi(v)$ consists of the (reduced) hyperplane
\[ \{ p\!: V \surj L \;|\; p(v) = 0 \} \subset \P(V). \]
\end{lem}

\begin{proof}
True by definition.
\end{proof}

\subsection{Non-free linear systems}

We will need a criterion for the general member of a linear system to be smooth, even
though that linear system has basepoints. First we need to define a notation:

\begin{ntn}\label{ntn:exterior-tensor-product}
  Let $X_1,X_2$ be two normal varieties and $D_1, D_2$ divisors on $X_1$ and $X_2$
  respectively.  We will use the notation $(D_1, D_2)$ to denote the ``exterior
  tensor product'' divisor $\pr_1^*D_1 + \pr_2^*D_2$ on $X_1\times X_2$, where
  $\pr_1,\pr_2$ are the natural projections to $X_1$ and $X_2$.
\end{ntn}

\begin{prp}[Linear systems on a product]\label{prp:lin sys product}
  Let $S$ be a smooth projective surface, and $B$ a smooth projective curve.  Let
  $|D_1|, |D_2|$ be linear systems on $S$ and on $B$, respectively. Assume that the
  scheme-theoretic base locus of $|D_1|$ consists of a single reduced closed point
  $q\in S$, while $|D_2|$ is basepoint-free.  Then a general element $T \in |(D_1,
  D_2)|$, $T \subset S \x B$, is smooth.
\end{prp}

The proof relies on the following easy lemma.

\begin{lem}[Image of a smooth divisor]\label{lem:image smooth}
  Let $C \subset X$ be a smooth curve in a smooth threefold. Let $\pi\!: \wt X \to X$
  be the blowup of $X$ along $C$, and $F_p = \pi^{-1}(p)$ the fibre over $p \in C
  \subset X$.  If $D \subset \wt X$ is a smooth divisor such that $D \cdot F_p = 1$
  and $F_p \not\subset D$ for all $p \in C$, then $D_X := \pi(D) \subset X$ is also
  smooth.
\end{lem}

\begin{proof}
  The assumption on the intersection between $D$ and $F_p$ implies that $\pi|_D:D\to
  D_X$ is an isomorphism.
\end{proof}

\begin{proof}[Proof of Proposition~\ref{prp:lin sys product}]
  Let $\pi'\!: \wt S \to S$ be the blowup of $S$ at $q$, with exceptional divisor
  $E$, and set $\pi = \pi' \x \id\!: \wt S \x B \to S \x B$. Then we have a
  decomposition into movable and fixed part
  \[ 
  |\pi^* (D_1, D_2)| = |M_0| + (E \x B), 
  \] 
  where $|M_0|$ is basepoint-free and one may choose a general element $M \in |M_0|$,
  such that $T = \pi(M)$. We will apply Lemma~\ref{lem:image smooth} to conclude that
  $T$ is smooth.

  By Bertini's theorem \cite[Ch.~III, Cor.~10.9]{Har77}, $M$ is smooth.  Furthermore,
  $M \cdot F_p = -(E \x B) \cdot F_p = -E^2 = 1$.  Finally, consider the divisor
  $M|_{E \x B}$. By Bertini again, it is also smooth. Hence, if it contained a fibre
  $F_p$, it would have to be a finite union of such fibres and then its image under
  $\pi$ would be finite. But we clearly have
  \[ 
  \pi(M|_{E \x B}) = \{ q \} \x B, 
  \] 
  leading to a contradiction. This shows that the assumptions of Lemma~\ref{lem:image
    smooth} are satisfied and hence the proof of Proposition~\ref{prp:lin sys
    product} is complete.
\end{proof}

\subsection{Further ancillary results}

We need two more lemmas, one of them about connectedness properties of big and nef
divisors and the other one about certain divisors on curves of genus $2$.

\begin{lem}[Connectedness of big and nef divisors]\label{lem:big nef conn}
  Let $X$ be a smooth projective variety of dimension $\ge 2$ and $D \subset X$ the
  support of an effective big and nef divisor. Then $D$ is connected.
\end{lem}

\begin{proof}
  Let $D'$ be an effective big and nef divisor with $\supp D' = D$.  By
  Kawamata-Viehweg vanishing, $H^1(X, \O_X(-D')) = 0$. Hence from the ideal sheaf
  sequence of $D' \subset X$ we obtain a surjection $H^0(X, \O_X) \surj H^0(D,
  \O_{D'})$, which implies that $H^0(D, \O_{D'}) = \C$, and therefore $D$ must be
  connected.
\end{proof}

\begin{lem}[Theta characteristics]\label{lem:genus 2}
  Let $B$ be a smooth projective curve of genus $g = 2$. Then there exists on $B$ a
  divisor $\Theta$ of degree $1$ with the following properties:
  \begin{enumerate}
  \item\label{seq:genus 2.2} $2\Theta \sim K_B$ is a canonical divisor,
  \item\label{seq:genus 2.1} $h^0(B, \Theta) = h^1(B, \Theta) = 0$,
  \item\label{seq:genus 2.4} $h^0(B, n\Theta)\neq 0$ and $h^1(B,n\Theta)=0$ for
    $n\geq 3$, and
  \item\label{seq:genus 2.3} the linear system $|n\Theta|$ is basepoint-free for
    $n\geq 3$.
  \end{enumerate}
\end{lem}

\begin{proof}
  The canonical linear system $|K_B|$ defines a two-to-one cover $B \to \P^1$, which,
  by the Hurwitz formula, is ramified at exactly $6$ points $R_1, \dots, R_6$. The
  $R_i$ are the only points with the property that $2R_i \sim K_B$.  On the other
  hand, we have, up to linear equivalence, $2^{2g} = 16$ divisors $\Theta$ satisfying
  $2\Theta \sim K_B$. So there exist, up to linear equivalence, $10$ divisors
  $\Theta$ such that $h^0(B, \Theta) = 0$ and $2\Theta \sim K_B$, and we choose one
  of them. This implies (\ref{lem:genus 2}.\ref{seq:genus 2.2}).

  By choice $h^0(B, \Theta)=0$, so (\ref{lem:genus 2}.\ref{seq:genus 2.1}) follows
  from Riemann-Roch on $B$. 
  
  By Serre duality and (\ref{lem:genus 2}.\ref{seq:genus 2.2}) it follows that
  $h^1(B,n\Theta)=h^0(B,(2-n)\Theta)=0$ for $n\geq 3$, and then by Riemann-Roch
  \begin{equation}
    \label{eq:8}
    h^0(B,n\Theta)=n-1, 
  \end{equation}
  so (\ref{lem:genus 2}.\ref{seq:genus 2.4}) follows.  

  Again, by Serre duality and (\ref{lem:genus 2}.\ref{seq:genus 2.2}) it follows that
  for any $P\in B$, $h^1(B,n\Theta-P)=h^0(B,(2-n)\Theta+P)$. This is clearly $0$ for
  $n\geq 4$, but also for $n=3$ since $\Theta\not\sim P$ by choice.  Then, again, by
  Riemann-Roch
  \begin{equation}
    \label{eq:99}
    h^0(B,n\Theta-P)=n-2, 
  \end{equation}
  so combining (\ref{eq:8}) and (\ref{eq:99}) we have that
  \[ 
  h^0(B, n\Theta - P) < h^0(B, n\Theta) 
  \] 
  for any point $P \in B$ and $n\geq 3$.  This proves (\ref{lem:genus
    2}.\ref{seq:genus 2.3}).
\end{proof}

\section{Proof of Theorem~(\ref{thm:pot DB}.\ref{itm:pot DB.3})}

\noindent
Theorem~(\ref{thm:pot DB}.\ref{itm:pot DB.3}) follows from the following more precise
result.

\begin{thm}[Non-Du~Bois \pdb variety with $\Q$-Cartier $K_X$]\label{thm:non-DB
    ambient}
  There is a $3$-dimensional Du~Bois pair $(X, \Sigma)$ such that $X$ is normal and
  has an isolated singularity such that $4K_X$ is Cartier and $\Sigma \subset X$ is a
  Weil divisor, but $X$ is not Du~Bois.
\end{thm}

The construction of $X$ can be outlined as follows. First we consider a ruled surface
$S$ and a section $E \subset S$ with suitable cohomological properties according to
Theorem~\ref{thm:cone DB criterion}. From $S$, we obtain a similar example $F \subset
T$, where additionally $T$ is of general type. The pair $(X, \Sigma)$ is then defined
to be a cone over $(T, F)$.

\begin{prp}[Ruled surface example]\label{prp:surface}
  There is a smooth projective surface $S$, a smooth curve $E \subset S$, and an
  ample divisor $L$ on $S$, such that:
  \begin{enumerate}
  \item\label{seq:surface 1} For all $i, n \ge 1$, it holds that $h^i(S, 4nL - E) =
    0$. 
  \item\label{seq:surface 2} We have $h^1(S, 4L) \ne 0$.
  \item\label{seq:surface 3} The divisor $5L - K_S$ is big and nef. The
    scheme-theoretic base locus of the linear system $|5L - K_S|$ consists of a
    single reduced point $b$, which does not lie on $E$.
  \item\label{seq:surface 5} For $n \ge 2$, we have $h^2(S, K_S + (4n - 5)L - E) =
    0$.
  \end{enumerate}
\end{prp}

\begin{prp}[Canonically polarized example]\label{prp:general type}
  There is a smooth projective surface $T$, a smooth (not necessarily irreducible)
  curve $F \subset T$, and an ample divisor $M$ on $T$, such that the following hold.
  \begin{enumerate}
  \item\label{seq:general type 1} For all $i, n \ge 1$, it holds that $h^i(T, nM - F)
    = 0$.
  \item\label{seq:general type 2} We have $h^1(T, M) \ne 0$.
  \item\label{seq:general type 3} The surface $T$ is canonically polarized. More
    precisely, $4K_T \sim 5M$.
  \end{enumerate}
\end{prp}

\subsection{Proof of Proposition~\ref{prp:surface}}

Since this proof is somewhat lengthy, it is divided into 7 steps.

\subsubsection*{Step 1: A hyperelliptic curve}

Let $C$ be a hyperelliptic curve of genus $g = 7$. By \cite[Ch.~IV,
Prop.~5.3]{Har77}, there exists on $C$ a unique $g^1_2$, that is, a linear system of
dimension $1$ and degree $2$.  The linear system $g^1_2$ defines a two-to-one cover
$f\!: C \to \P^1$, ramified at exactly $16$ points $R_1, \dots, R_{16}$.

Let $H \sim 2g^1_2$, and let $A = R_1$ be one of the ramification points of $f$.
Note that $2A \sim g^1_2$, and that
\begin{sequation}\label{eq:4(H-A)} 
  4(H - A) \sim 6g^1_2 \sim K_C
\end{sequation}%
is a canonical divisor on $C$ by \cite[Ch.~IV, Prop.~5.3]{Har77}.

We need to calculate $f_* \O_C$. Since a torsion-free sheaf on a smooth curve is
locally free, and since the injection $\O_{\P^1} \inj f_* \O_C$ is split by the trace
map, we must have $f_* \O_C = \O_{\P^1} \oplus \O_{\P^1}(n)$ for some integer $n$. To
calculate $n$, observe that
\[ 
H^1(C, f^* \O_{\P^1}(6)) = H^1(\P^1, f_* f^* \O_{\P^1}(6)) = H^1(\P^1, \O_{\P^1}(6)
\oplus \O_{\P^1}(n + 6)) 
\] 
by the Leray spectral sequence and the projection formula. Using~\eqref{eq:4(H-A)},
we see that the left-hand side is one-dimensional.  So $n + 6 = -2$, and
\begin{sequation}\label{eqn:f_* O_C}
  f_* \O_C = \O_{\P^1} \oplus \O_{\P^1}(-8).
\end{sequation}%

\subsubsection*{Step 2: Construction and properties of $S$}

Let $\pi\!: S = \P_C(\O_C \oplus \O_C(-A)) \to C$ and let $\sO_S(1)$ denote the
associated relatively ample line bundle. Take $E \subset S$ to be the section of
$\pi$ corresponding to the surjection $\O_C \oplus \O_C(-A) \surj \O_C(-A)$, and take
$E_\infty \subset S$ to be the section corresponding to the projection onto the
first summand $\O_C \oplus \O_C(-A) \surj \O_C$.  We will identify divisors on $E$,
on $E_\infty$, and on $C$. Notice that $E$ is contained in the linear system
$|\O_S(1)|$, that $\pi_*\sO_S(E)\simeq \sO_C\oplus\sO_C(-A)$, and that $\O_S(E)|_E =
\O_E(-A)$.  See also Notation~\ref{ntn:split ruled surface}.

\subsubsection*{Step 3: Definition and ampleness of $L$}

Let $L = \pi^* H + E$. We use the Nakai-Moishezon criterion to show that $L$ is
ample.  Let $D \subset S$ be an irreducible curve. If $D = E$ or $D$ is a fibre of
$\pi$, then $L \cdot D = 3$ or $1$, respectively, hence we may assume that neither is
the case. Then $\pi^* H \cdot D > 0$ and $E \cdot D \ge 0$, so $L \cdot D > 0$ again.
Finally observe that $L^2 = 7 > 0$. So $L$ is ample.

\subsubsection*{Step 4: Proof of~(\ref{prp:surface}.\ref{seq:surface 1})}

Let $i, n \ge 1$. Since $4nL - E = \pi^*(4nH) + (4n - 1)E$, this divisor intersects
the fibres of $\pi$ non-negatively, so $R^1\pi_* \O_S(4nL - E)=0$ and we obtain
\[ 
H^i(S, \sO_S(4nL - E)) \isom H^i(C, \pi_* \O_S(4nL - E)). 
\] 
This already proves~(\ref{prp:surface}.\ref{seq:surface 1}) if $i \ge 2$, hence we
may assume that $i = 1$.  Using the identification $\sO_S(E)\simeq \sO_S(1)$ we have
\[ 
\pi_* \O_S(4nL - E) = \O_C(4nH) \tensor \pi_* \O_S(4n - 1) = \bigoplus_{k=0}^{4n-1}
\O_C(4nH - kA).
\] 
For $n, k$ in the relevant range, the degree of $4nH - kA$ is at least $16n - (4n -
1) = 12n + 1 \ge 13 > \deg K_C = 12$.  So we have $h^1(C, 4nH - kA) = 0$ by Serre
duality. This proves~(\ref{prp:surface}.\ref{seq:surface 1}).

\subsubsection*{Step 5: Proof of~(\ref{prp:surface}.\ref{seq:surface 2})}

Consider the short exact sequence
\[ 
0 \to \O_S(4L - E) \to \O_S(4L) \to \O_E(4L|_E) \to 0 
\] 
and its associated long exact sequence. In view
of~(\ref{prp:surface}.\ref{seq:surface 1}), we obtain an isomorphism
\[
H^1(S, \sO_S(4L)) \isom H^1(E, \sO_E(4L|_E)).
\] 
On the other hand $4L|_E \sim 4(H - A) \sim K_E$ by \eqref{eq:4(H-A)}.  Since $h^1(E,
K_E)=1$, (\ref{prp:surface}.\ref{seq:surface 2}) follows.

\subsubsection*{Step 6: Proof of~(\ref{prp:surface}.\ref{seq:surface 3})}

We have
\begin{equation*}
  \begin{array}{lcl}
    5L - K_S &=   & \pi^*(5H) + 5E - ( -2E + \pi^*(K_C - A) ) \\
             &=   & \pi^*(5H + A - K_C) + 7E \\
             &\sim& \pi^*(4g^1_2 + R_1) + 7E.
  \end{array}
\end{equation*}
By the projection formula, we obtain an isomorphism
\[ 
H^0(S, \sO_S(5L - K_S)) = \bigoplus_{k=0}^7 H^0(C, \sO_C(4g^1_2 + (1 - k)R_1)).
\]
In order to
prove the claim about the base locus, we will repeatedly apply
Proposition~\ref{prp:div ruled} to produce sufficiently many divisors in $|5L -
K_S|$.

First, since $|4g^1_2 + R_1|$ has no basepoint outside $R_1$, taking $k = 0$ we see
that the base locus of $|5L - K_S|$ is contained in $\pi^{-1}(R_1) \cup E$.  On the
other hand, since $|4g^1_2 - 6R_1| = |g^1_2|$ is basepoint-free, taking $k = 7$ shows
that the base locus in question is contained in $E_\infty$.  Hence the only basepoint
of $|5L - K_S|$ can be at the intersection $\pi^{-1}(R_1) \cap E_\infty = \{ b \}$.
In particular, $5L - K_S$ is nef. Calculating that $(5L - K_S)^2 = 77 > 0$, shows
that $5L - K_S$ is also big.

In order to see that $\Bs |5L - K_S|$ is reduced, we will exhibit two members of $|5L
- K_S|$ smooth at $b$, with different tangent directions.  To this end, note that
$|4g^1_2|$ is basepoint-free, so taking $k = 1$ gives us a member of $|5L - K_S|$
which is smooth at $b$, with tangent space equal to $T_b E_\infty \subset T_b S$. But
we have already seen that $k = 0$ gives a member of $|5L - K_S|$ smooth at $b$, with
tangent space equal to $T_b (\pi^{-1}(R_1)) \subset T_b S$.

It remains to show that $b$ really is a basepoint of $|5L - K_S|$.  Because $(5L -
K_S)|_{E_\infty} \sim 4g^1_2 + R_1$, it is enough to show that $R_1$ is a basepoint
of $|4g^1_2 + R_1|$.  Note that $h^0(C, 4g^1_2) = 5$ and $h^0(C, 5g^1_2) = 6$ by the
projection formula and~\eqref{eqn:f_* O_C}. Since $|5g^1_2|$ is basepoint-free, we
have $h^0(C, 4g^1_2 + R_1) = h^0(C, 5g^1_2 - R_1) = 5$. So $h^0(C, 4g^1_2) = h^0(C,
4g^1_2 + R_1)$. This shows the claim and finishes the proof
of~(\ref{prp:surface}.\ref{seq:surface 3}).

\subsubsection*{Step 7: Proof of~(\ref{prp:surface}.\ref{seq:surface 5})}

Let $n \ge 2$. By Serre duality,
\[ h^2(S, K_S + (4n - 5)L - E) = h^0(S, -(4n - 5)L + E). \]
The right-hand side vanishes since $-(4n - 5)L + E$ intersects a fiber of $\pi$ negatively.
This proves~(\ref{prp:surface}.\ref{seq:surface 5}) and also finishes the
proof of Proposition~\ref{prp:surface}.

\subsection{Proof of Proposition~\ref{prp:general type}}

Again, we divide the proof into 5 steps.

\subsubsection*{Step 1: Construction of $T$}

Let $S$, $E$, and $L$ be as in Proposition~\ref{prp:surface}, and let $B$ and
$\Theta$ be as in Lemma~\ref{lem:genus 2}. Consider the product $X = S \x B$ with
projections $\pr_1$ and $\pr_2$. We will utilize the notation from
\eqref{ntn:exterior-tensor-product}. Let $T \subset X$ be a general element of the
linear system $|(5L - K_S, 3\Theta)|$.  Let $F$ be the divisor $(\pr_1^*E)|_T$, and
let $M = (4L, 4\Theta)|_T$.  Clearly $M$ is ample.

\subsubsection*{Step 2: $T$ is smooth}

We have $\Bs |5L - K_S| = \{ b \}$ scheme-theoretically, and $\Bs |3\Theta| = \emptyset$. So by Proposition~\ref{prp:lin sys product}, $T$ is
smooth. Also note that $T$ is connected by Lemma \ref{lem:big nef conn}.

Since $b \not\in E$, the restricted linear system $|(5L - K_S, 3\Theta)|_{\pr_1^{-1}(E)}$ is basepoint-free. Hence $F \subset T$ is a smooth curve.

\subsubsection*{Step 3: Proof of~(\ref{prp:general type}.\ref{seq:general type 1})}

Consider the ideal sheaf sequence of $T \subset X$,
\begin{equation*}
0 \to \O_X(K_S - 5L, -3\Theta) \to \O_X \to \O_T \to 0,
\end{equation*}%
and twist it by $(4nL - E, 4n\Theta)$, yielding
\begin{sequation}\label{seq:nM-F} 
  0 \to \O_X(K_S + (4n-5)L - E, (4n-3)\Theta) \to \O_X(4nL - E, 4n\Theta) \to \O_T(nM
  - F) \to 0.
\end{sequation}%
When calculating the cohomology groups of an exterior tensor product using the
K\"unneth formula, we will always drop any summands where at least one factor
vanishes simply for dimension reasons.  Concerning the sheaf in the middle, we have
\begin{equation*}
  h^i(\O_X(4nL - E, 4n\Theta)) = h^{i-1}(S, 4nL - E) \cdot \underbrace{ h^1(B,
    4n\Theta) }_{\text{$= 0$, (\ref{lem:genus 2}.\ref{seq:genus 2.4})}} 
  + \underbrace{ h^i(S, 4nL - E) }_{\text{$= 0$, (\ref{prp:surface}.\ref{seq:surface
        1})}} \cdot h^0(B, 4n\Theta) 
\end{equation*}
for all $i \ge 1$, and for the sheaf on the left-hand side,
\begin{multline*}
  h^i(\O_X(K_S + (4n-5)L - E, (4n-3)\Theta)) = \\
  = h^{i-1}(S, K_S + (4n-5)L - E) \cdot \underbrace{ h^1(B, (4n-3)\Theta)
  }_{\text{$= 0$, (\ref{lem:genus 2}.\ref{seq:genus 2.1}/\ref{seq:genus 2.4})}} + \\
  + \underbrace{ h^i(S, K_S + (4n-5)L - E) }_{\text{$= 0$ if $n \ge 2$,
      (\ref{prp:surface}.\ref{seq:surface 5})}} \cdot \underbrace{ h^0(B,
    (4n-3)\Theta) }_{\text{$=0$ if $n=1$, (\ref{lem:genus 2}.\ref{seq:genus 2.1})}}
\end{multline*}
for $i \ge 2$. So taking cohomology of \eqref{seq:nM-F} proves (\ref{prp:general
  type}.\ref{seq:general type 1}).

\subsubsection*{Step 4: Proof of~(\ref{prp:general type}.\ref{seq:general type 2})}
Twist the ideal sheaf sequence of $T \subset X$ by $(4L, 4\Theta)$, which gives
\begin{sequation}\label{seq:M}
  0 \to \O_X(K_S - L, \Theta) \to \O_X(4L, 4\Theta) \to \O_T(M) \to 0.
\end{sequation}%
We have

\begin{equation*}
  h^1(\O_X(K_S - L, \Theta)) = h^0(S, K_S - L) \cdot \underbrace{ h^1(B, \Theta)
  }_{\text{$= 0$, (\ref{lem:genus 2}.\ref{seq:genus 2.1})}} 
  + \; h^1(S, K_S - L) \cdot \underbrace{ h^0(B, \Theta). }_{\text{$= 0$,
      (\ref{lem:genus 2}.\ref{seq:genus 2.1})}} 
\end{equation*}
Furthermore,
\begin{equation*}
  h^1(\O_X(4L, 4\Theta)) \ge
  \underbrace{ h^1(S, 4L) }_{\text{$\ne 0$, (\ref{prp:surface}.\ref{seq:surface 2})}}
  \cdot \; 
  \underbrace{ h^0(B, 4\Theta), }_{\text{$\ne 0$, (\ref{lem:genus 2}.\ref{seq:genus
        2.4})}} 
\end{equation*}
so taking cohomology of \eqref{seq:M} gives $h^1(T, M) \ne 0$, proving
(\ref{prp:general type}.\ref{seq:general type 2}).

\subsubsection*{Step 5: Proof of~(\ref{prp:general type}.\ref{seq:general type 3})}

Since $3\Theta = 5\Theta - K_B$, by the adjunction formula we have
\[ K_T = (K_X + T)|_T = (5L, 5\Theta)|_T. \] Claim (\ref{prp:general
  type}.\ref{seq:general type 3}) follows immediately.  This finishes the proof of
Proposition \ref{prp:general type}.

\subsection{Proof of Theorem~\ref{thm:non-DB ambient}}

Let $T$, $F$, and $M$ be as in Proposition~\ref{prp:general type}. By
Theorem~\ref{thm:cone DB criterion}, $(X, \Sigma) = C_a(T, F, \O_T(M))$ is a Du~Bois
pair, but $(X, \emptyset) = C_a(T, \emptyset, \O_T(M))$ is not. This means that $X$
is not Du~Bois. By \cite[Prop.~3.14]{SingBook}, $4K_X$ is a Cartier divisor and this
is the smallest multiple of $K_X$ which is Cartier.  It is clear by construction that
$X$ is normal, of dimension three, and has an isolated singularity.

\section{Proof of Corollary~\ref{cor:LZ pdb}}

Let $X$ be a variety satisfying the assumption of the Lipman-Zariski conjecture. By
\cite[Thm.~3]{Lip65}, $X$ is normal.  Corollary~\ref{cor:LZ pdb} now follows
immediately from Theorem~\ref{thm:Ext of p-forms on DB pairs}
and~\cite[Thm.~1.2]{GK13}.

Alternatively, we may also argue as follows: If $X$ is \pdb and satisfies the
assumption of the Lipman-Zariski conjecture, then $X$ is log canonical by
Theorem~\ref{thm:pot DB}.\ref{itm:pot DB.2}. By~\cite[Thm.~1.1]{Dru13}, $X$ is
smooth.



\providecommand{\bysame}{\leavevmode\hbox to3em{\hrulefill}\thinspace}
\providecommand{\MR}{\relax\ifhmode\unskip\space\fi MR}
\providecommand{\MRhref}[2]{%
  \href{http://www.ams.org/mathscinet-getitem?mr=#1}{#2}
}
\providecommand{\href}[2]{#2}

\end{document}